\documentclass[draft]{amsart}
\usepackage{amsmath,amsfonts,amssymb,amscd,verbatim,multicol}
\usepackage[arrow,matrix,cmtip]{xy}

\usepackage{fullpage}

\newcommand{\Ext}{\operatorname{Ext}}
\newcommand{\trdeg}{\operatorname{tr.deg}}

\newcommand{\GKdim}{\operatorname{GKdim}}

\newcommand{\bpr}{\begin{proof}}
\newcommand{\epr}{\end{proof}}

\newcommand{\spec}{\operatorname{Spec}}

\newcommand{\mc}{\mathcal}

\newcommand{\mb}{\mathbb}

\newcommand{\wt}{\widetilde}

\numberwithin{equation}{section}
 \theoremstyle{plain}
\newtheorem{theorem}[equation]{Theorem}

\newtheorem{lemma}[equation]{Lemma}

\newtheorem{corollary}[equation]{Corollary}
\newtheorem{proposition}[equation]{Proposition}

\theoremstyle{definition}
\newtheorem{definition}[equation]{Definition}

\newtheorem{example}[equation]{Example}

\author[Rogalski]{Daniel Rogalski}
\address{University of California, San Diego, Department of Mathematics, 9500 Gilman Dr. \# 0112, La Jolla, CA 92093-0112} 
\email{drogalski@ucsd.edu}

\thanks{The author was partially supported by the NSA grant H98230-15-1-0317.}
\date{\today}

\title{Stably noetherian algebras of polynomial growth}

\begin{document}

\begin{abstract}
Let $A$ be a right noetherian algebra over a field $k$.  If the base field extension $A \otimes_k K$ remains right noetherian for all extension fields $K$ of $k$, then $A$ is called \emph{stably right noetherian} over $k$.  We develop an inductive method to show that certain algebras of finite Gelfand-Kirillov dimension are stably noetherian, using critical composition series.  We use this to characterize which algebras satisfying a polynomial identity are stably noetherian.  The method also applies to many $\mb{N}$-graded rings of finite global dimension; in particular, we see that a noetherian Artin-Schelter regular algebra must be stably noetherian.  In addition, we study more general variations of the stably noetherian property where the field extensions are restricted to those of a certain type, for instance  purely transcendental extensions. 
\end{abstract}

\maketitle

\section{Introduction}

Let $R$ be a (unital, not necessarily commutative) algebra over a field $k$.  It is often advantageous to extend the base field $k$ to some larger field $K$ and consider the $K$-algebra $R \otimes_k K$ instead.  For example, one may want to work over a base field with a large cardinality, or one which is algebraically closed.  One naturally wishes for important properties of $R$ to be preserved.   In this paper we study the basic question of when the noetherian property still holds after base field extension.   In particular, we consider the following notion first introduced by Farina \cite{Fa}.
\begin{definition}
\label{def:sn1}
Let $R$ be an algebra over a field $k$.  The algebra $R$ is
called \emph{stably right noetherian} (over $k$) if $R \otimes_k
K$ is right noetherian for all extension fields $K$ of $k$. 
\end{definition}
\noindent Of course, there are also the obvious left-sided or two-sided variants, and the definition can just as easily be made for $R$-modules.

The behavior of the noetherian property under tensor products more generally has long been known to be a tricky issue, even for commutative algebras.   In one of the seminal papers on the subject, V{\'a}mos studied tensor products of two fields over a common subfield \cite{Va}. In particular, he proved the following important result.
\begin{theorem} \cite[Theorem 11]{Va}
\label{thm:vamos}
Let $k \subseteq K$ be a field extension.  Then $K \otimes_k K$ is noetherian if and only if $K$ is finitely generated 
as a field extension of $k$.
\end{theorem}
\noindent
Thus, although any field $K$ is noetherian, if it is an infinitely generated field extension of its subfield $k$, then 
it will not be stably noetherian over $k$.  It is easy to explain the main idea behind the proof.  If $K$ is an infinitely generated extension of $k$, then there is an infinite proper ascending chain of subfields $k = K_0 \subsetneq K_1 \subsetneq K_2 \subsetneq \dots \subsetneq K_n \subsetneq \dots \subseteq K$.  One checks that if $I_n$ is the kernel of the natural surjective map $K \otimes_k K \to K \otimes_{K_n} K$, then $\{I_n\}$ gives an infinite proper ascending chain of ideals in $K \otimes_k K$.

In this paper, a $k$-algebra $R$ will be called \emph{affine} (over $k$) when it is finitely generated as a $k$-algebra.  It is clear that if $R$ is affine over $k$ and $k \subseteq K$ is an extension field, then 
$R \otimes_k K$ is affine over $K$.  Any affine commutative $k$-algebra is noetherian, by the Hilbert basis theorem.  It follows that any affine commutative $k$-algebra $R$ is stably noetherian over $k$.  It is also easy to 
see that localization preserves the stably noetherian property, so localizations of affine commutative $k$-algebras 
must be stably noetherian over $k$.  This explains the other, easier implication of V{\'a}mos's theorem.

On the other hand, Resco and Small gave an example of a noncommutative affine $k$-algebra which is noetherian but not stably noetherian, in \cite{RS}.  It is an Ore extension of the form $S = E[s; \delta]$, where $E$ is an infinitely generated purely inseparable algebraic extension of a base field $k$ (in particular, $E$ has characteristic $p > 0$), but where nonetheless $S$ is affine over $k$.  V{\'a}mos's result implies that $E$ is not stably noetherian over $k$, and this then implies that $S$ is not stably noetherian over $k$ either, since $S$ is a faithfully flat $E$-module.  

The strongest result in the positive direction is due to Bell.
\begin{theorem} \cite[Theorem 1.2]{Be}
Let $k$ be an uncountable, algebraically closed base field.  If $A$ is a right noetherian countably generated $k$-algebra, then $A$ is stably right noetherian. 
\end{theorem}
\noindent
Of course, Bell's theorem still leaves open the case where one is working over an arbitrary base field, and wants to extend it, for example, to an algebraically closed one, without losing the noetherian property.  

In this paper, we will study the stably noetherian property over arbitrary base fields, but instead restrict more heavily the kinds of algebras we consider.   In particular, we show how a hypothesis of finite Gelfand-Kirillov (GK) dimension, or in other words polynomial growth, can be used to prove that certain algebras are stably noetherian.  The main idea is to induct on the dimension, using critical composition series. 

We begin in Section~\ref{sec:dim}  with a review some of the basic facts about GK dimension which we will need.  
Recall that a right $A$-module $M$ with $\GKdim(M) = d$ is called $d$-homogeneous if every nonzero submodule $N$ of $M$ 
has $\GKdim(N) = d$ as well.  The module $M$ is called $d$-critical if $\GKdim(M/N) < d$ for all such nonzero submodules $N$.
A critical composition series of $M$ is a series of submodules 
$0 = M_0 \subsetneq M_1 \subsetneq \dots \subsetneq M_{n-1} \subsetneq M_n = M$ such that each $M_{i+1}/M_i$ is $d_i$-critical, 
where $d_i \leq d_{i+1}$ for all $i$.  Our results will only apply to algebras $A$ for which GK dimension of finitely generated $A$-modules is exact, finitely partitive, and integer valued.  We will review the definitions below, but the important point is that these properties tend to hold for modules over reasonable algebras one encounters in practice (possibly even all noetherian algebras), and so should be considered weak assumptions.  These properties 
are needed to ensure that critical composition series of noetherian modules exist and behave reasonably.
In Section~\ref{sec:dim} we also study the behavior of the homogeneous and critical properties of modules under base field extension.  We show that the homogeneous property is preserved by a base field extension, but the critical property is not in general (though it is if the extension is purely transcendental).

Next, in Section~\ref{sec:def} we give the basic definitions and results concerning the stably noetherian property.
In fact, we work somewhat more generally.  Because one might be willing to restrict the kinds of base field extensions one does (for example, sticking to purely transcendental extensions), we set up a general definition of stably noetherian, which refers to the preservation of the noetherian property under all field extensions of some specified kind.   
Our main inductive method follows in Section~\ref{sec:ind}.  Its basic idea is as follows.  Since criticality is not preserved by base field extension, in general the length of a critical composition series of a module grows after base field extension.  However, 
in some cases one is able to show for all noetherian $A$-modules $M$ that there is a bound to how much the length of a critical composition series of $M \otimes_k K$ can grow, for all extensions $k \subseteq K$ of the kind of interest.  When this happens, an inductive argument on the dimension can be used to prove that $M$ is stably noetherian for that kind of extension.   

In the remainder of the paper we give three applications of the method.  Since purely transcendental extensions preserve criticality, 
in this case our inductive method applies easily to give a quite general result.  
\begin{theorem}[Theorem~\ref{thm:purely-noeth}]
Let $R$ be a right noetherian $k$-algebra with $\GKdim_k R < \infty$.  For each $d \geq 0$ write 
$R_d = R \otimes_k k(x_1, \dots, x_d)$.  Assume that for each $d \geq 0$ the GK dimension of right $R_d$-modules 
over the field $k(x_1, \dots, x_d)$ is exact, finitely partitive, and takes integer values.  
Then $R \otimes_k K$ is right noetherian for all purely transcendental field extensions $k \subseteq K$.
\end{theorem}

Next, in Section~\ref{sec:PI} we study the special case of polynomial identity (PI) rings.  As already mentioned, 
affine commutative $k$-algebras are stably noetherian.  This result was extended to affine right noetherian PI $k$-algebras 
by Small \cite{Sm}; see \cite[Proposition 2.51]{Fa}.  However, it is still interesting to ask which possibly non-affine PI algebras are stably right noetherian.  Our method allows us to settle this, in particular answering \cite[Question 2.55]{Fa}.
\begin{theorem}[Corollary~\ref{cor:PI-char}]
Let $R$ be a right noetherian PI $k$-algebra.  Then $R$ is stably right noetherian over $k$ if and only if for 
all prime ideals $P$ of $R$, the center $Z$ of the classical ring of fractions $Q(R/P)$ is a finitely generated field extension of $k$. 
\end{theorem}

Finally, in Section~\ref{sec:graded} we consider the special case of $\mb{N}$-graded $k$-algebras 
$R = \bigoplus_{n \geq 0} R_n$ which are locally finite, that is $\dim_k R_n < \infty$  for all $n \geq 0$.
We focus on such $R$ which are \emph{homologically smooth} over $k$.  The general definition is given in that section, but here we just note that 
 in the important special case that $R_0 = k^m$ for some $m$ (in particular if $R$ is \emph{connected}, that is $R_0 = k$),
then homological smoothness is equivalent to finite global dimension, by a result from \cite{RR1}.

Graded $R$-modules over a homologically smooth algebra of finite GK dimension have well-behaved Hilbert series, and 
therefore a well-defined multiplicity.  The multiplicity limits the growth of the length of critical composition series under base field 
extension, and this allows us to apply our inductive technique.  
\begin{theorem}[Theorem~\ref{thm:graded-stable}]
\label{thm:hs}
Let $R$ be a locally finite right noetherian $\mb{N}$-graded homologically smooth $k$-algebra.
Then $R$ is stably right noetherian over $k$.
\end{theorem}
In particular, our result applies to the important case of graded twisted Calabi-Yau algebras, which are homologically smooth as part of their definition (see Section~\ref{sec:graded}).
\begin{corollary}
\label{cor:cy}
Let $R$ be a locally finite right noetherian $\mb{N}$-graded twisted Calabi-Yau algebra.  Then $R$ is stably 
right noetherian over $k$.
\end{corollary}
The class of graded twisted Calabi-Yau algebras includes, as a special case when the algebra is connected, 
the well-known class of Artin-Schelter (AS) regular algebras.  Note that we do not need to assume $\GKdim(R) < \infty$ 
in Theorem~\ref{thm:hs}, as it is implied by the noetherian hypothesis in this case.   

We hope that the main idea of our inductive method presented in Section~\ref{sec:ind} might have other applications.  
For example, in the theory of Hilbert schemes it is useful to consider graded algebras $R$ which are \emph{strongly noetherian} in the sense that $R \otimes_k C$ is noetherian for all commutative noetherian $k$-algebras $C$.   Stonger yet, one says $R$ is \emph{universally noetherian} if $R \otimes_k S$ is noetherian for all noetherian $k$-algebras $S$.  See \cite{ASZ} for more details.  We do not attempt to study these properties here, but we wonder if our method has applications to the study of strongly or universally noetherian algebras.   It could be, for example, that noetherian locally finite graded twisted Calabi-Yau algebras are always universally noetherian.

Another interesting open question relates to the example of Resco and Small of an affine right noetherian but not stably right noetherian algebra.   A claimed example by Medvedev with similar properties but over a field of characteristic $0$, which was mentioned in a note in \cite{RS}, was never published.  We recently viewed a draft of Medvedev's argument which appeared in correspondence with Resco, but it does not appear to be correct.  It would be interesting to settle one way or the other whether there can be an affine noetherian but not stably noetherian algebra in characteristic $0$.

\section*{acknowledgments}
We thank John Farina and Lance Small for many interesting and helpful conversations.

\section{GK dimension, critical composition series, and base field extension}
\label{sec:dim}

The main results of this paper work for certain algebras of finite Gelfand-Kirillov (GK) dimension.  
See \cite{KL} for more background on the following basic facts about  GK dimension.
Recall that for a $k$-algebra $R$, we define $\GKdim_k(R) = \sup_V \left(\limsup_{n \to \infty} \log_n(\dim_k V^n)\right)$, where $V$ runs over 
all finite-dimensional $k$-subspaces of $R$.  If the base field $k$ is understood, we write this as $\GKdim(R)$.
If $R$ is affine over $k$, then for any finite-dimensional subspace $V$ containing $1$ 
which generates $R$ as an algebra, in fact we have $\GKdim_k(R) =  \limsup_{n \to \infty} \log_n(\dim_k V^n)$.
For any right $R$-module $M = M_R$ over the  $k$-algebra $R$, we define 
\[
\GKdim_k (M_R) = \sup_{W, V} \left(\limsup_{n \to \infty} \log_n(\dim_k W V^n)\right),
\]
 where $W$ runs over finite dimensional $k$-subspaces of $M$ and $V$ over finite-dimensional $k$-subspaces of $R$.
Again, if $M$ is finitely generated as a right $R$-module and $R$ is affine, in fact we have 
$\GKdim_k(M) = \limsup_{n \to \infty} \log_n(\dim_k W V^n)$, for any $W$ containing a generating set of $M$ and any 
finite-dimensional subspace $V$ of $R$ which contains $1$ and generates $R$ as an algebra.  
It is convenient for us to make the convention $\GKdim_k(0) = -1$ for the zero module.

It is standard that for any short exact sequence $0 \to M \to N \to P \to 0$ of right $R$-modules, $\GKdim(M) \leq \GKdim(N)$ and $\GKdim(P) \leq \GKdim(N)$.  GK dimension is called \emph{exact} for right $R$-modules if for all such exact sequences we in fact have $\GKdim(N) = \max(\GKdim(M), \GKdim(P))$.  It is easy to see that it is equivalent to require this condition for short exact sequences of finitely generated modules.  While examples of $k$-algebras for which GK dimension is not exact are known, one generally has exactness for examples one encounters in practice.  In particular, to our knowledge no noetherian $k$-algebra $R$ is known for which exactness fails.  We will assume 
that we have a $k$-algebra $R$ with exact GK dimension for the rest of this section, since without this basic property the theory of critical modules and critical composition series which we want to use does not work well.

Let $M$ be a right $R$-module with $\GKdim(M) < \infty$.  
We call $M$ \emph{$d$-critical} if $\GKdim(M) = d$ while $\GKdim(M/N) < d$ for all nonzero submodules $N$ of $M$.  The module $M$ is \emph{$d$-homogeneous} if $\GKdim(N) = d$ for all nonzero submodules $N$ of $M$.  
A $d$-critical module $M$ is $d$-homogeneous (using exactness) and in fact even uniform, in other words, any two nonzero submodules of $M$ have nonzero intersection.  Also, a nonzero submodule of a $d$-critical module is again $d$-critical.
A \emph{critical composition series} for an $R$-module $M$ with $\GKdim_k(M) < \infty$ is a finite filtration of submodules  $M_0 = 0 \subseteq M_1 \subseteq M_2 \subseteq \dots \subseteq M_n = M$, such that the corresponding factors $M_i/M_{i-1}$ are $d_i$-critical for all $1 \leq i \leq n$, and where $d_i \leq d_{i+1}$ for all $1 \leq i \leq n-1$.  The number $n$ is called the \emph{length} of the critical composition series, and 
we also write $\ell(M)$ for $n$.  By standard arguments, any two critical composition series of $M$ have the same length, and there is some pairing of their factors so that paired factors contain nonzero isomorphic submodules.  See the second half of \cite[Theorem 15.9]{GW}, which is stated for Krull dimension but whose proof works also for GK dimension (again, using exactness).

On the other hand, unlike the case of Krull dimension treated in \cite[Section 15]{GW}, it is not clear that a noetherian module 
$M$ must have a critical composition series for GK dimension, without a further assumption. 
We  say that GK dimension of $R$-modules is \emph{finitely partitive} if given any finitely generated right 
module $N$ with $\GKdim(N) = d$, there is an upper bound $n_0$ (depending on $N$) on the lengths $n$ of descending chains 
$N =N_0 \supsetneq N_1 \supsetneq N_2 \supsetneq \dots \supsetneq N_n$ with $\GKdim(N_i/N_{i+1}) = d$ for all 
$0 \leq i \leq n-1$.  If GK dimension of $R$-modules is finitely partitive (and exact), then every noetherian right $R$-module $M$ with 
$\GKdim(M) < \infty$ has a critical 
composition series, by the same proof as for the Krull dimension treated in the first half of \cite[Theorem 15.9]{GW}, using 
\cite[Lemma 15.8]{GW}.  In fact, it suffices to know that for each noetherian $R$-module $N$ there are no \emph{infinite} descending chains $N =N_0 \supsetneq N_1 \supsetneq N_2 \supsetneq \dots$ with $\GKdim(N_i/N_{i+1}) = \GKdim(N)$ 
for all $i \geq 0$.  Again, to our knowledge there are no known examples of noetherian $k$-algebras of finite GK dimension for which GK dimension is not finitely partitive.  

Here are some basic facts about critical composition series which we will find useful below.
\begin{lemma}
\label{lem:ccfacts}
Let $R$ be a $k$-algebra and assume that GK dimension of finitely generated $R$-modules is exact and finitely partitive.  Let $0 \to M \to P \to N \to 0$ be an exact sequence of noetherian right $R$-modules with $\GKdim_k(P) < \infty$.
\begin{enumerate}
\item $\ell(M) \leq \ell(P)$.
\item $\ell(P) \leq \ell(M) + \ell(N)$.
\end{enumerate}
\end{lemma}
\begin{proof}
(1)  Given a critical composition series $0 = P_0 \subsetneq P_1 \subsetneq \dots \subsetneq P_n = P$ for $P$, where 
$P_i/P_{i-1}$ is $d_i$-critical, then 
setting $M_n = P_n \cap M$, the factor $M_i/M_{i-1}$ is isomorphic to a submodule of $P_i/P_{i-1}$, and hence is either $0$ 
or else is also $d_i$-critical.  Thus after removing repeats, the series given by $\{M_n \}$ gives a critical composition 
series for $M$, which has length at most $\ell(P)$.

(2)  Identify $M$ with a submodule of $P$ and $N$ with $P/M$.  We induct on $\ell(M)$.  The result is trivial if $M = 0$.  Suppose that $\ell(M) = 1$, in other words $M$ is $d$-critical for some $d$.  Let $N'$ be the unique largest submodule of $N$ such that $\GKdim(N') < d = \GKdim(M)$ (it is possible that $N' = 0$).   Write $N' = L/M$ for some $M \subseteq L \subseteq P$.  Note that $\GKdim(L) = d$ by exactness
and that $P/L \cong N/N'$ has no submodules of GK dimension less than $d$.  Thus a critical composition series of $L$ can be followed by  
a critical composition series of $N/N'$ to obtain a critical composition series of $P$, so 
$\ell(P) = \ell(N/N') + \ell(L)$.  Similarly, a critical composition series of $N'$ followed by a critical composition series of $N/N'$ gives 
a critical composition series of $N$, so $\ell(N) = \ell(N') + \ell(N/N')$.  Thus it suffices to show that $\ell(L) \leq \ell(M) + \ell(N')$.
By Zorn's Lemma, we can choose a submodule $Q \subseteq L$ maximal with respect to $Q \cap M = 0$.  Then $Q$ is isomorphic to a submodule of $N'$, so $\ell(Q) \leq \ell(N')$ by part (1).  By the definition of $Q$, $M \cong (M+Q)/Q \subseteq L/Q$ is easily seen to be an essential extension of $M$, extended by the module $L/(M + Q)$, which is a homomorphic image of $N'$, and so 
has GK dimension less than $d$.  Since $M$ is $d$-critical, it follows that $L/Q$ is again $d$-critical.  Now a critical composition series for $Q$ followed by $L/Q$ gives  a
critical composition series for $L$, so $\ell(L) = 1 + \ell(Q) \leq 1 + \ell(N') = \ell(M) + \ell(N')$ as required.

To complete the induction step, let $\ell(M) = n$ for some $n \geq 2$, and suppose the result is true whenever $\ell(M) \leq n-1$.
Letting $0 = M_0 \subseteq M_1 \subseteq \dots \subseteq M_n = M$ be a critical composition series for $M$, we have a short exact sequence $0 \to M/M_1 \to P/M_1 \to P/M \to 0$ for which $\ell(M/M_1) = n-1$ and so $\ell(P/M_1) \leq \ell(M/M_1) + \ell(P/M) = n-1 + \ell(P/M)$ by the induction hypothesis.   We also have a short exact sequence $0 \to M_1 \to P \to P/M_1 \to 0$ to which the induction hypothesis applies
and so $\ell(P) \leq 1 + \ell(P/M_1) \leq 1 + n-1 + \ell(P/M) = n + \ell(P/M) = \ell(M) + \ell(P/M)$.
\end{proof}

In some cases, central localization preserves the property of a module being critical.
\begin{lemma}
\label{lem:loc-crit}
Let $R$ be a $k$-algebra with $\GKdim_k(R) = d < \infty$,
 and let $X$ be a multiplicative system of nonzero central regular elements in $R$.  
\begin{enumerate}
\item Let $M$ be an $X$-torsionfree right $R$-module.  Then $\GKdim_k(M_R) = \GKdim_k((MX^{-1})_{RX^{-1}})$.
\item Let $M$ be a finitely generated $d$-critical right $R$-module.  Then $MX^{-1}$ is a $d$-critical $RX^{-1}$-module.
\end{enumerate}
\end{lemma}
\begin{proof}
(1) Since the elements in $X$ are central, $X$ is automatically a right denominator set in $R$.  
We assume that $M$ is $X$-torsionfree, in other words that if $m \in M$ and $x \in X$ with $mx = 0$, then $m = 0$.  
This implies that the natural $R$-module map $M \to MX^{-1}$ is an embedding, and so 
$\GKdim_k(M_R) \leq \GKdim_k((MX^{-1})_R) \leq \GKdim_k((MX^{-1})_{RX^{-1}})$.  The reverse inequality follows 
from a similar argument as in \cite[Proposition 4.2]{KL}, except applied in a module instead of an algebra.

(2) We claim that $M$ is $X$-torsionfree.  Suppose not, and choose $0 \neq m \in M$ such that $mx = 0$ for some $x \in X$.  Since $M$ is $d$-critical, it is $d$-homogeneous, so $\GKdim_k( mR) = d$.  On the other hand, $mR \cong R/J$ for the right ideal 
$J = \operatorname{r.ann}_R(m)$.  Since $x$ is a regular element in $R$, $\GKdim_k(R/xR) < \GKdim_k(R) = d$ by 
\cite[Proopsition 5.1(e)]{KL}.  Then since $x \in J$, $\GKdim_k(mR) = \GKdim_k(R/J) \leq \GKdim_k(R/xR) < d$, a contradiction.  
So $M$ is $X$-torsionfree as claimed.  Thus $\GKdim_k((MX^{-1})_{RX^{-1}}) = d$ by part (1).  Every nonzero $RX^{-1}$-submodule of $MX^{-1}$ has the form $NX^{-1}$, where $N$ is a nonzero $R$-submodule of $M$ such that $M/N$ is $X$-torsionfree 
\cite[Theorem 10.15]{GW}.  Then $MX^{-1}/NX^{-1} \cong (M/N)X^{-1}$, and since $\GKdim_k(M/N) < d$, we have 
$\GKdim_k((M/N)X^{-1}) < d$ by part (1).  So $MX^{-1}$ is $d$-critical.
\end{proof}

Our aim in the rest of this section is to investigate the interaction of GK dimension with extension of the base field.  
Suppose that $M$ is a right module over the $k$-algebra $R$.  If $k \subseteq K$ is any field extension, we 
may consider the $K$-algebra $R_K = R \otimes_k K$.  We identify $R$ with the subalgebra 
$\{ (r \otimes 1) | r \in R \} \subseteq R \otimes_k K$.   We have the extended right $R_K$-module $M_K = M \otimes_k K$, which we can also think of as an $R$-module by restriction.

We have the following standard lemma.  We omit the proof, which follows easily from the definitions of GK dimension.
\begin{lemma}
\label{lem:GKbasics}
Let $R$ be a $k$-algebra and $M$ a right $R$-module.  Let $k \subseteq K$ be a field extension.
\begin{enumerate}
\item $\GKdim_k(M_R) = \GKdim_K( (M_K)_{R_K})$.
\item Suppose that $[K:k] < \infty$.  Then for any right module $N$ over $R_K$ we have $\GKdim_k(N_R) = \GKdim_k(N_{R_K}) = \GKdim_{K}(N_{R_K})$.
\end{enumerate}
\end{lemma}

The previous result shows that base field extension preserves the value of the GK dimension of a module.  If 
$M$ is a homogeneous or critical module, however, it is less obvious when one should expect this property to be preserved.  In fact, it is easy to find examples where base field extension destroys criticality; for example, 
the $\mb{R}$-algebra $\mb{C}$ is a $0$-critical module over itself but $\mb{C} \otimes_{\mb{R}} \mb{C} \cong \mb{C} \oplus \mb{C}$ is not critical.  We see next, however, that the homogeneous property is always preserved by base field extension.
\begin{proposition}
\label{prop:homogeneous}
Let $R$ be a $k$-algebra and let $M$ be a right $R$-module. 
Let $k \subseteq K$ be a field extension, and let $R_K = R  \otimes_k K$ and $M_K = M \otimes_k K$. 
If $M$ is $d$-homogeneous $R$-module with respect to GK dimension over $k$, then $M_K$ is a $d$-homogeneous $R_K$-module for GK dimension over $K$.
\end{proposition}
\begin{proof}
 Let $\GKdim_k(M_R) = d$, so $\GKdim_K((M_K)_{R_K}) = d$ by Lemma~\ref{lem:GKbasics}.  
Suppose that $M_K$ fails to be $d$-homogeneous.  Then there is 
$z = \sum_{i=1}^p m_i \otimes a_i \subseteq M \otimes_k K$ such that $\GKdim_K ((z R_K)_{R_K}) < d$.
Now $\ell = k(a_1, \dots, a_p) \subseteq K$, where $k \subseteq \ell$ is 
a finitely generated field extension.  We can think of 
$z \in M \otimes_k \ell$, and hence we get
\[
z R_K = z(R \otimes_k K) = z (R_{\ell} \otimes_{\ell} K) = (z R_{\ell}) \otimes_{\ell} K.
\]
Thus $\GKdim_K ((z R_K)_{R_K}) = \GKdim_{\ell}((zR_{\ell})_{R_{\ell}}) < d$ by Lemma~\ref{lem:GKbasics}, and so $M_{\ell} = M \otimes_k \ell$ also fails to 
be $d$-homogeneous over $R_{\ell}$.  

Changing notation back, we may assume from now on that $k \subseteq K$ is a finitely generated field extension.  By induction, it is enough to consider the cases where $k \subseteq K$ is an extension with $[K: k] < \infty$, or where $k \subseteq K = k(x)$ is a purely transcendental extension in one variable.    Consider first 
the case where $[K: k] < \infty$.
Let $N$ be a nonzero $R_K$-submodule of $M_K$.  Then we have $\GKdim_k(N_R) = \GKdim_{K}(N_{R_K})$ by 
Lemma~\ref{lem:GKbasics}(2).  Moreover, as a right $R$-module, $M_K = M \otimes_k K \cong \bigoplus_{i=1}^{[K:k]} M$.  It is 
easy to see that direct sums of $d$-homogeneous modules are again $d$-homogeneous, so 
$M_K$ is a $d$-homogeneous right $R$-module.  Then $d = \GKdim_k(N_R) = \GKdim_{K}(N_{R_K})$, and 
so $M_K$ is a $d$-homogeneous right $R_K$-module, as required.

Now consider the case $k \subseteq K = k(x)$.  Again, suppose that $M_K$ is not $d$-homogeneous.  As above, we can then 
choose $0 \neq z = \sum_{i=0}^m n_i \otimes f_i(x) \in M \otimes_k K$ such that $\GKdim_K( (z R_K)_{R_K}) < d$.
If $g(x) \in k[x]$ is a common denominator for the rational functions $f_i(x)$, replacing $z$ with $z (1 \otimes g(x))$, we can assume that $z \in M \otimes_k k[x]$, without changing $z R_K$.  Now we can write $z = \sum_{i=0}^p m_i \otimes x^i \subseteq M \otimes_k k[x]$, with 
$m_p \neq 0$.  Assume we have chosen such a nonzero $z$ with a minimal number of nonzero $m_i$.  For each nonzero $m_i$ consider $J_i = \operatorname{r.ann}_R(m_i)$, so that $m_i R \cong R/J_i$ as right $R$-modules.  If $J_i \neq J_j$ for some $i, j$, say with $J_i \nsubseteq J_j$, then we can choose $r \in R$ with $r \in J_i \setminus J_j$.  
Then $z' = z (r \otimes 1) \in z R_K$, so that $\GKdim_K( (z' R_K)_{R_K}) < d$, while $0 \neq z'$ has fewer nonzero summands, a contradiction.
Thus all of the $J_i$ are equal to a single right ideal $J$.
 
Certainly $J \otimes_k K \subseteq \operatorname{r.ann}_{R_K} (z)$, and we claim that this is an equality.  
Suppose not 
and choose $w \in \operatorname{r.ann}_{R_K} (z) \setminus J \otimes_k K$.  By clearing denominators we may assume that 
$w \in R \otimes_k k[x]$.   Write $w = \sum_{i=0}^q r_i \otimes x^i$ where $r_q \neq 0$.  We may choose such a $w$ 
with $q$ as small as possible.  If $m_p r_q \neq 0$,  the term $m_p r_q \otimes x^{p +q}$ in $zw$ is nonzero and cannot be cancelled by any other term, contradicting $zw = 0$.  This forces 
$m_p r_q = 0$, and thus $r_q \in J$.  Now subtracting $r_q \otimes x^q$ from $w$ we get a new $w$ with smaller $q$, 
a contradiction.  Thus $J \otimes_k K = \operatorname{r.ann}_{R_K} (z)$ as claimed.   But this means that 
$z R_K \cong (R \otimes_k K)/(J \otimes_k K) \cong (R/J) \otimes_k K$.  
Thus 
\[
\GKdim_{K}((zR_K)_{R_K}) = \GKdim_{K}((R/J) \otimes_k K) = \GKdim_k((R/J)_R) = d, 
\]
using Lemma~\ref{lem:GKbasics}(1), and since $R/J$ is isomorphic to a nonzero submodule $m_p R$ of the $d$-homogeneous $R$-module $M$.  This contradicts our choice of $z$.  Thus $M_K$ is $d$-homogeneous in this case as well.  
\end{proof}

As a useful consequence of the stability of the homogeneous property under base field extensions, we see that the length 
of a critical composition series for a module can only grow after base field extension.
\begin{lemma}
\label{lem:cs-ext}
Let $R$ be a $k$-algebra and let $k \subseteq K$ be a field extension.  Assume that GK dimension of $R$-modules and $R_K$-modules is exact.  Suppose that $M$ is a right $R$-module with a critical composition series $M_0 = 0 \subseteq M_1 \subseteq M_2 \subseteq \dots \subseteq M_n = M$ (for GK dimension over $k$).  Suppose that for all $1 \leq i \leq n$, the $R_K$-module $(M_i/M_{i-1}) \otimes_k K$ has a critical composition series (for GK dimension over $K$).  Then $M_K = M \otimes_k K$ has a critical composition series which is a refinement of the series $0 \subseteq M_1 \otimes_k K \subseteq M_2 \otimes_k K \subseteq \dots \subseteq M_n \otimes_k K = M_K$.  In particular, the length of a critical composition series of $M_K$ is greater than or equal to the length of the critical composition series for $M$.
\end{lemma}
\begin{proof}
By definition, each $M_i/M_{i-1}$ is $d_i$-critical, where $d_i \leq d_{i+1}$ for all $i$.   In particular, 
each factor $M_i/M_{i-1}$ is $d_i$-homogeneous, so by Proposition~\ref{prop:homogeneous}, $(M_i/M_{i-1}) \otimes_k K$ 
is still $d_i$-homogeneous for GK dimension over $K$.   Choosing a critical composition series for each $(M_i/M_{i-1}) \otimes_k K$ (which is assumed to exist), it is easy to see that all of its factors must have dimension $d_i$ as well.   
Thus stitching these critical composition series together yields a critical composition series for $M \otimes_k K$, which is a refinement of the series with terms $M_i \otimes_k K$ by construction.
\end{proof}

While we have noted that criticality is not stable under base field extension in general, by using a similar idea as in the last part of the proof of Proposition~\ref{prop:homogeneous}, we can prove that criticality is preserved under a transcendental base field extension.
\begin{proposition}
\label{prop:critical}
Let $R$ be a $k$-algebra and $M$ a right $R$-module.  Let $k \subseteq K = k(x)$ where $k(x)$ 
is a function field in one variable.   Assume that GK dimension is exact for right $R$ and $R_K$-modules, and finitely partitive for $R_K$-modules.  If $M$ is a critical $R$-module for GK dimension over $k$, then $M$ is a critical $R_K$-module for GK dimension over $K$.
\end{proposition}
\begin{proof}
Let $M$ be $d$-critical over $R$, so it is also $d$-homogeneous.  Suppose that $M_K$ fails to be $d$-critical, 
and let $N$ be a nonzero $R_K$-submodule of $M_K$ such that $\GKdim_{K}(M_K/N) = \GKdim_{K}(M_K)$ ( $=\GKdim_k M = d$).

We first claim  for any $R$-submodule $0 \neq M' \subseteq M$ that $M'_K$ also fails to be $d$-critical.  
We have $\GKdim_k((M/M')_R) < \GKdim_k(M_R)$ since $M$ is critical, and so using Lemma~\ref{lem:GKbasics}(1) we have 
\begin{gather*}
\GKdim_{K}((M_K/M'_K)_{R_K}) = \GKdim_K (((M/M') \otimes_k K)_{R_K})  = \GKdim_k ((M/M')_R)    \\
< \GKdim_k(M_R) = \GKdim_{K}((M_K)_{R_K}). 
\end{gather*}
Since $\GKdim_K(M_K/N) = d$ we must have $\GKdim_K(M_K/(N \cap M'_K)) = d$.  Then from the chain $N \cap M'_K \subseteq M'_K \subseteq M_K$, by exactness of GK dimension for $R_K$ it follows that $\GKdim_K(M'_K/(N \cap M'_K)) = d$.
Note that $N/(N \cap M'_K)$ is isomorphic to a submodule of $M_K/M'_K$, which 
has GK dimension less than that of $M_K$, but $N$ has the same GK dimension as $M_K$ since $M_K$ is homogeneous 
by Proposition~\ref{prop:homogeneous}.  This implies that $N \cap M'_K \neq 0$.  
Thus the submodule $N \cap M'_K$ of $M'_K$ demonstrates that $M'_K$ also fails to be critical, as claimed.

Next, we claim that any nonzero submodule $0 \neq P \subsetneq M_K = M \otimes_k k(x)$ must also fail to be $d$-critical.
By picking any nonzero $z = \sum_{i=0}^p m_i \otimes x^i \in P$ with a minimal number of nonzero coefficients and following the same argument as in the last part of Proposition~\ref{prop:homogeneous}, we see that $P$ has a nonzero cyclic submodule $P' = z R_K \cong (R/J) \otimes_k K$, for some right ideal $J$ of $R$ such that $R/J$ is isomorphic to a nonzero $R$-submodule $M'$ of $M$.  
In particular, by the previous paragraph $M'_K \cong P'$ is not $d$-critical, so $P$ is not $d$-critical either, proving the claim.

Now since every nonzero submodule of $M_K$ has GK dimension $d$ but fails to be $d$-critical, by induction we can choose an infinite 
sequence of nonzero submodules $M_K = N_0 \supsetneq N_1 \supsetneq N_2 \supsetneq \dots$ 
such that $\GKdim_K(N_i/N_{i+1}) = d$ for all $i \geq 0$.  This contradicts the assumption that GK dimension of $R(x)$-modules is finitely partitive (in fact we don't even need the full strength of partitivity, just the lack of infinite descending chains with factors all of the same dimension).    Thus $M_K$ must be $d$-critical.
\end{proof}

 \section{Stably noetherian algebras}
\label{sec:def}

A s mentioned in the introduction, in some cases one is content to extend a base field by some more 
special kind of extension.  For example, if one only wants to ensure the base field has a large enough cardinality for some purpose, it would be enough to take a purely transcendental extension, which tends to have a milder effect on the properties of an algebra than an algebraic extension of the base field.  To account for this situation, we will study a more general version of the stably noetherian property than that defined in Definition~\ref{def:sn1}.

Let $\mc{F}$ stand for some class of field extensions $F \subseteq K$.   One choice of course is the class of all field extensions, 
or for another example, $\mc{F}$ could be the class of purely transcendental extensions, namely those of the form  $F \subseteq K$ where $K \cong F(x_{\alpha} | \alpha \in I)$ for some set of indeterminates $\{ x_{\alpha} \}$ indexed by $I$.
\begin{definition}
Let $R$ be a right noetherian $k$-algebra.  Given a class $\mc{F}$ of field extensions, we say that 
$R$ is \emph{$\mc{F}$-stably right noetherian} if $R \otimes_k K$ is right noetherian for all extensions $k \subseteq K$ in the class $\mc{F}$.  When $\mc{F}$ is the class of all field extensions we simply say that $R$ is \emph{stably right noetherian}.  We write $\mc{P}\mc{T}$ for the class of purely transcendental extensions, so the algebras for which the noetherian property is stable under extensions in this class are called \emph{$\mc{P}\mc{T}$-stably right noetherian}.
\end{definition}
\noindent

The notion of $\mc{F}$-stably noetherian works best when the class of extensions $\mc{F}$ has some mild closure properties.
For the purposes of this paper only, we make the following definition which simply includes all of the properties we need to 
make our later results go through.  
\begin{definition}
Let $\mc{F}$ be a class of field extensions.  We say that $\mc{F}$ is \emph{extensive} if it satisfies the following 
properties:  (i) given $k \subseteq L$ and $L \subseteq K$ in the class, then $k \subseteq K$ is in the class; 
(ii) given $k \subseteq L$ and $k \subseteq L'$ in the class, there is an extension $k \subseteq E$ in the class with 
subfields $k \subseteq F \subseteq E$, $k \subseteq F' \subseteq E$, and $k$-isomorphisms $L \cong F$ and $L' \cong F'$; and 
(iii)  If $k \subseteq K$ is in the class and $k \subseteq L \subseteq K$ and $k \subseteq L' \subseteq K$ are subextensions
with $k \subseteq L$ and $k \subseteq L'$ finitely generated, then there is a subextension 
$k \subseteq L L' \subseteq M \subseteq K$ with $k \subseteq M$ again finitely generated, and with 
$k \subseteq M$ and $M \subseteq K$ in the class.
\end{definition}
\noindent
Roughly speaking, the first property shows that being in the class $\mc{F}$ is ``closed under extensions", and the 
second property says that any two extensions of a field which are in the class can be both embedded inside a single extension in the class.  The third more technical property will be used in the proof of the main inductive method in the next section.

Suppose that $k \subseteq L$ and $k \subseteq L'$ are any two field extensions.  Recall that we can always embed them up to isomorphism in one extension $E$.  Namely, if $P$ is any maximal ideal of $L \otimes_k L'$, then $E = L \otimes_k L'/P$ is a field and there are 
canoncial (necessarily injective) $k$-algebra homomorphisms from $\{ (x \otimes 1) | x \in L\} \cong L$ and $\{ (1 \otimes y) | y \in L' \} \cong L'$ to $E$.  In fact, the composite of the images of these homomorphisms is all of $E$.  Thus any class of extensions which is closed under 
composites satisfies property (ii), and if it is also closed under subextensions (i.e. if given $k \subseteq L \subseteq K$ and $k \subseteq K$ 
in the class, then both $k \subseteq L$ and $L \subseteq K$ are in the class), then it also satisfies (iii) by taking $M = L L'$.   In this way it follows from standard theorems of field theory that familiar classes $\mc{F}$ such as algebraic extensions, separable algebraic extensions, or purely inseparable algebraic extensions are extensive.   

The proof that the class $\mc{F} = \mc{P} \mc{T}$ of purely transcendental extensions is extensive is a bit different.  Property (i) is a well-known result, and for (ii), simply choose a purely transcendental extension in a number of variables exceeding the cardinalities of the transcendence bases of the two given extensions.  For (iii), choose some  transcendence basis $\{ x_{\alpha} \}$ so that $K = k(x_{\alpha} | \alpha \in I)$.  Then given $L = k(a_1, \dots, a_n) \subseteq K$ and $L' = k(b_1, \dots, b_m) \subseteq K$, note that we may choose a finite subset $J$ of $I$ so that all $a_i$ and $b_i$ are contained in $M = k(x_{\beta} | \beta \in J)$.

We study next how the $\mc{F}$-stably noetherian property behaves under some common constructions.
\begin{lemma}
\label{lem:stableprops}
Let $R$ be a right noetherian $k$-algebra, and let $\mc{F}$ be an extensive class of field extensions.
\begin{enumerate}
\item If $R$ is $\mc{F}$-stably right noetherian over $k$, then so is every factor ring of $R$, every localization of $R$ at a right denominator set, any Ore extension $R[x; \sigma, \delta]$ with automorphism $\sigma$ and $\sigma$-derivation $\delta$ (over $k$), and any ring extension $S$ where $R \subseteq S$ and $S$ is finitely generated as a right $R$-module.
\item If $R \subseteq S$ is a extension of $k$-algebras such that $S$ is faithfully flat as a left $R$-module (in particular, if $S$ is a free left $R$-module), then if $S$ is $\mc{F}$-stably right noetherian over $k$, so is $R$.
\item If $k \subseteq \ell$ is a field extension in the class $\mc{F}$ and $R$ is an $\ell$-algebra, then if $R$ is $\mc{F}$-stably right noetherian over $k$ it is also $\mc{F}$-stably right  noetherian over $\ell$.
\item If $k \subseteq \ell$ is an extension in the class $\mc{F}$, then $R$ is $\mc{F}$-stably right noetherian over $k$ if and only if $R \otimes_k \ell$ is $\mc{F}$-stably right noetherian over $\ell$.
\end{enumerate}
\end{lemma}
\begin{proof}
(1)  The proofs are the same as for the usual stably noetherian property; see \cite[Propositions 2.16, 2.22, 2.26, 2.28]{Fa}.
\begin{comment}
Suppose that $R$ is $\mc{F}$-stably noetherian, so that $R \otimes_k K$ is noetherian for all field extensions 
$k \subseteq K$ in the class $\mc{F}$.

If $R/I$ is a factor ring of $R$, then $R/I \otimes_k K \cong (R \otimes_k K)/(I \otimes_k K)$ is a factor 
ring of $R \otimes_k K$ and hence is also noetherian for all $K$.

If $RX^{-1}$ is a localization of $R$ at the multiplicative system $X$, then it is well known that since 
$R$ is noetherian, so is $RX^{-1}$.  Now $RX^{-1} \otimes_k K$ is a localization of $R \otimes_k K$ 
at the multiplicative system $\{ (x \otimes 1) | x \in X \}$, and hence is also noetherian for all $K$.

We have $R[x; \sigma, \delta] \otimes_k K \cong (R \otimes_k K)[x; \sigma', \delta']$, where $\sigma', \delta'$ 
are the extensions to $R \otimes_k K$ satisfing $\sigma'(r \otimes 1) = \sigma(r) \otimes 1$ and $\delta'(r \otimes 1) 
= \delta(r) \otimes 1$.  The latter ring is also noetherian by the Hilbert Basis Theorem.

Let $R \subseteq S$ with $S$ finitely generated as a left $R$-module, say.  Then $R \otimes_k K \subseteq S \otimes_k K$ is an extension which is also finite on the left, and so $S \otimes_k K$ is left noetherian.  Since $S$ is also 
finitely generated as a right $R$-module, we also get that $S \otimes_k K$ is right noetherian.
\end{comment}

(2)  This follows from the fact that if $S$ is a faithfully flat left $R$-module then the right noetherian 
property for $S$ passes down to $R$, together with the fact that faithful flatness is preserved by 
base field extension.  See \cite[Lemma 2.42, Proposition 2.43]{Fa}.

(3)  If $R$ is $\mc{F}$-stably right noetherian over $k$, then for any extension $\ell \subseteq K$ in the class, $k \subseteq K$ is also in the class
by property (i) of the definition of extensive, so $R \otimes_k K$ is right noetherian.  This ring surjects onto $R \otimes_{\ell} K$, so the latter ring is also right noetherian.  Thus $R$ is $\mc{F}$-stably right noetherian over $\ell$.

(4) Suppose that $R$ is $\mc{F}$-stably right noetherian over $k$.  Then for any extension $\ell \subseteq K$ in the class, 
$(R \otimes_k \ell) \otimes_{\ell} K \cong R \otimes_k K$ is right noetherian, since $k \subseteq K$ is also in the class. 
 So $R \otimes_k \ell$ is $\mc{F}$-stably right noetherian over $\ell$.

Conversely, if $R \otimes_k \ell$ is $\mc{F}$-stably right noetherian over $\ell$, then consider any extension 
$k \subseteq K$ in the class $\mc{F}$.  By property (ii) of the definition of extensive, there is a field $E$ with extensions $\ell \subseteq E$ 
and $K \subseteq E$ in the class (possibly replacing $\ell$ and $K$ with isomorphic copies).  By assumption, $(R \otimes_k \ell) \otimes_{\ell} E \cong R \otimes_k E$ is noetherian.  Then $R \otimes_k K \subseteq R \otimes_k E$ is a free and hence faithfully flat left-module extension, so $R \otimes_k K$ is also right noetherian, by the same argument as in part (2).  Thus $R$ is $\mc{F}$-stably right noetherian over $k$.
\end{proof}

By a standard reduction, one can reduce to the case of prime rings in studying the stably noetherian property, as follows.
\begin{lemma}
\label{lem:prime-reduce}
If $R$ is a right noetherian $k$-algebra, then $R$ is $\mc{F}$-stably right noetherian if and only if $R/P_i$ is $\mc{F}$-stably right noetherian for all minimal primes $P_i$.
\end{lemma}
\begin{proof}
This is an easy consequence of the fact that the prime radical $N$ of $R$ is nilpotent together with the fact that $R/N$ embeds subdirectly in 
$\bigoplus_{i=1}^n R/P_i$.  The proof is given in \cite[Proposition 2.21, Proposition 2.34, Remark 2.35]{Fa}.
\end{proof}
\begin{comment}
(1).  If each $R/P_i$ is noetherian, note that $P_1 \cap \dots \cap P_m = 0$ since $R$ is semiprime, 
and the natural map $R \mapsto R/P_1 \oplus \dots \oplus R/P_m$ is an embedding (of rings or of $R$-modules).  It follows that $R$ is a noetherian $R$-module and hence a noetherian ring.  The converse is 
trivial.

(2).  Let $P_1, \dots P_m$ be the minimal primes of $R$, of which there are finitely many since $R$ is noetherian, and let $N = P_1 \cap \dots \cap P_m$ be the nilradical of $R$.   One direction is trivial.

Thus assume now that $R/P_i$ is $\mc{F}$-stably noetherian for all $i$.  Since $N$ is nilpotent and finitely generated, $N' = N \otimes_k K$ is finitely generated as 
an ideal of $R \otimes_k K$ and nilpotent.  It easily follows that as long as $R \otimes_k K /N'$ is noetherian, 
then so is $R \otimes_k K$.  Thus it is enough to prove that $(R/N) \otimes_k K$ is noetherian and in 
this way we may assume that $R$ is reduced.  Now we have the embedding $R \mapsto R/P_1 \oplus \dots \oplus R/P_m$ as in part (1).  Then this extends to an embedding 
$R \otimes_k K \to (R/P_1 \otimes_k K) \oplus \dots \oplus (R/P_m \otimes_k K)$.  Since each $R/P_i$ is 
stably noetherian, $R \otimes_k K$ is a subdirect product of a finite direct sum of noetherian rings and so is 
noetherian.
\end{comment}

\section{Inductive method}
\label{sec:ind}

In this section we prove our main technical result, which is a method for proving that certain 
rings are $\mc{F}$-stably noetherian by induction on the GK dimension of a module, provided one 
can verify that there is a bound to how much base field extension in the class increases the lengths of critical composition series.   Recall that for a right module $M$ over the $k$-algebra $R$ and a field extension $k \subseteq K$, we will write $M_K = M \otimes_k K$, which is a right $R_K = R \otimes_k K$-module.
\begin{theorem}
\label{thm:reduce}
Let $k$ be a field and let $R$ be a right noetherian $k$-algebra with $\GKdim_k R < \infty$.   Let $\mc{F}$ be some extensive class of field extensions.  Fix a field extension $k \subseteq K$ in the class, and call a field $L$ \emph{small} if $k \subseteq L \subseteq K$, with $k \subseteq L$ and $L \subseteq K$ in the class $\mc{F}$, and with $k \subseteq L$ a finitely generated field extension.  Suppose that 
\begin{enumerate}
\item[(i)] For all small fields $L$, GK dimension of  finitely generated $R_L$-modules is exact, 
finitely partitive, and integer valued.  
\item[(ii)]  Given any small field $L$ and a finitely generated right $R_L$-module, there is a upper bound $n = n(M)$ on the lengths of the 
$\GKdim_{L'}$-critical composition series for the modules $M_{L'} = M \otimes_L L'$, as $L'$ varies over all small fields with $L \subseteq L'$ in the class $\mc{F}$.
\end{enumerate}
Then $R \otimes_k K$ is right noetherian.  

In particular, if conditions (i) and (ii) hold for all extensions $k \subseteq K$ in the class, then $R$ is $\mc{F}$-stably right noetherian over $k$.
\end{theorem}
\begin{proof}
Fix the extension $k \subseteq K$ and let $L$ be a small field.  Note that $R_L$ is noetherian, and that condition (i) implies in particular that every finitely generated $R_L$-module has a critical composition series, with well-defined length.  
\begin{comment}
If $M$ is any finitely generated right $R_L$-module, then we claim that there is a upper bound $n = n(M)$ on the lengths of the $\GKdim_{L'}$-critical composition series for the modules $M_{L'} = M \otimes_L L'$ as $L'$ varies over $\mc{F}$-small fields with $L \subseteq L'$ in the class.  If $M$ is critical, the claim is just condition (ii).
In general, choose a $\GKdim_L$-critical composition series for $M$, say 
$0 = M_0 \subseteq M_1 \subseteq \dots \subseteq M_r = M$,
with $N_i = M_i/M_{i-1}$ critical for $1 \leq i \leq r$.   Then by Lemma~\ref{lem:cs-ext}, the series of $M'$ with terms $M_i \otimes_L L'$ can be refined to a $\GKdim_{L'}$-critical composition series for $M'$, by piecing together critical composition series for each $N_i \otimes_L L'$.  
Then the length of this series is at most $n(N_1) + \dots + n(N_r)$, so we can take this number for $n(M)$, proving the claim.
\end{comment}

The main idea of the proof is to show that for all noetherian right $R$-modules 
$M$, the extended module $M_K = M \otimes_k K$ is noetherian over $R_K$, by induction on the GK dimension of modules as well as on the length of critical composition series.   Let $\mb{Z}_{\geq m} = \{ a \in \mb{Z} | a \geq m \}$.  We define the well-ordered poset $\Gamma = (\mb{Z}_{\geq -1}) \times (\mb{Z}_{\geq 0})$, with lexicographic order where $(c, i) \leq (d, j)$ if $c< d$ or if $c = d$ and $i \leq j$.     Let $L$ be a small field.  If $M$ is a finitely generated right $R_L$-module, let $n(M)$ be as in (ii) (for definiteness, let $n(M)$ be 
the smallest possible bound for each $M$).  Then if the critical composition series for $M$ with respect to $\GKdim_L$ has length $m$, we define $D_L(M) = (\GKdim_L(M), n(M)-m)$.  This is a kind of generalized dimension which accounts for both the GK dimension of $M$, as well as 
how much freedom its critical composition series has to grow in length under base field extension.

We claim next that for any extension $L \subseteq E$ in the class, with $E$ small, then $n(M_E) = n(M)$.  For each extension $E \subseteq L'$ in the class with $L'$ small, $L \subseteq L'$ is also in the class and $M_E \otimes_E L' = M \otimes_L L'$; 
thus $n(M_E)$ depends on fewer fields and so $n(M_E) \leq n(M)$.  On the other hand, suppose that $L \subseteq \wt{L}$ is in the class with $\wt{L}$ also small, such that the critical composition series of $M_{\wt{L}}$ acheives the maximum possible critical composition series length $n(M)$.  By condition (iii) in the definition of extensive class, we can find $L \subseteq E\wt{L}  \subseteq F$ with $F$ still small.  By Lemma~\ref{lem:cs-ext}, the length of the critical composition series of $M_F = M_{\wt{L}} \otimes_{\wt{L}} F$ is at least as large as that of $M_{\wt{L}}$, so it is also the maximum value $n(M)$.  Since $F$ contains $E$, this shows $n(M_E) \geq n(M)$.  Thus $n(M_E) = n(M)$ as claimed.

Now for $\alpha \in \Gamma$, consider the following proposition $P(\alpha)$:  for every small field $L$ and every noetherian right $R_L$-module $M$ with $D_L(M) \leq \alpha$, then $M_K = M \otimes_L K$ is a noetherian $R_K$-module.   We prove that $P(\alpha)$ holds for all $\alpha \in \Gamma$ by (transfinite) induction on the well-ordered poset $\Gamma$.
Recall that we use the convention $\GKdim_L(0) = -1$, and only the zero module has dimension $-1$.
Since the zero module is trivially stably noetherian, $P(\alpha)$ is trivially true for any $\alpha = (-1, i)$.
This gives the base cases of the induction.

Now consider some $\alpha = (d, i) \in \Gamma$ with $d \geq 0$ and assume that $P(\beta)$ holds for all $\beta < \alpha$ in $\Gamma$.  Suppose we have a noetherian right $R_L$-module $M$ with $D_L(M) = \alpha$, where $L$ is small.  We must show that the $R_K = R_L \otimes_L K$-module $M \otimes_L K$ is noetherian.
  
Suppose first that there is some extension $L \subseteq L'$ in the class with $L'$ small, such that $M' = M \otimes_L L'$ has a $\GKdim_{L'}$-critical composition series of longer length $m'$ than the length $m$ of the $\GKdim_L$-critical composition series of $M$.   We saw above that $n(M) = n(M')$. Thus $D_{L'}(M') = (d, n(M) - m') < (d, n(M) - m) = D_{L}(M)$.
By the induction hypothesis, $M' \otimes_{L'} K$ is noetherian.  Thus $M \otimes_L K = (M \otimes_L L') \otimes_{L'} K = M' \otimes_{L'} K$ is noetherian as an $R_K$-module as required.  

From now on we may assume that we are not in the situation of the previous paragraph.  Thus given an extension $L \subseteq L'$ in the class 
$\mc{F}$ with $L'$ small, the $R_{L'}$-module $M' = M_{L'}$ has a critical composition series for GK dimension over $L'$ of the same length as the critical composition series of $M$.   Choose a $\GKdim_L$-critical composition series for $M$, say 
$0 = M_0 \subseteq M_1 \subseteq \dots \subseteq M_r = M$, with $N_i = M_i/M_{i-1}$ critical for $1 \leq i \leq r$.   Then by Lemma~\ref{lem:cs-ext}, the series of $M'$ with terms $M_i \otimes_L L'$ can be refined to a $\GKdim_{L'}$-critical composition series for $M'$.  
Since this series must have the same length, each $(M_i \otimes_L L')/(M_{i-1} \otimes_L L') \cong N_i \otimes_L L'$ must remain critical as an $R_{L'}$-module.   In other words, $n(N_i) = 1$ and so $D_L(N_i) = (\GKdim_L(N_i), 0)$.  Now to show that 
$M \otimes_L K$ is noetherian, it suffices to show that each $N_i \otimes_L K$ is noetherian.   If $\GKdim_L(N_i) < d$
then the induction hypothesis applies.

Thus we are left to deal with $\GKdim_L$-critical modules of dimension $d$ which remain critical under all base field extensions $L \subseteq L'$ in the class with $L'$ small.  Change notation back so that $M$ is such a module.   Let $N$ be a nonzero $R_K$-submodule of $M_K = M \otimes_L K$, and identify $M$ with the subset $\{ m \otimes 1 | m \in M \}$ of $M \otimes_L K$.   If $P = N \cap M \neq 0$, then $\GKdim_L(M/P) < d$ since $M$ is $d$-critical, and so $D_L(M/P) < D_L(M)$ and thus 
$(M/P) \otimes_L K$ is noetherian by the induction hypothesis.  Then $(M/P) \otimes_L K \cong (M \otimes_L K)/(P \otimes_L K)$ is noetherian, and so $N/(P \otimes_L K)$ is a finitely generated submodule of $(M \otimes_L K)/(P \otimes_L K)$.  On the other hand, $P$ is a finitely generated $R_L$-module since $M$ is a noetherian module,  and so $P \otimes_L K$ is a finitely generated $R \otimes_L K$-module.  It follows that $N$ is finitely generated over $R_K$.

Finally, if instead $P = N \cap M = 0$, choose some nonzero element $\sum_{i=1}^j m_i \otimes k_i \in M \otimes_L K$ which is in $N$.  
By property (iii) in the definition of extensive class, we can choose an extension $L \subseteq L'$ in the class with $L'$ small, 
such that  $k_i \in L'$ for all $i$.  Write $M \otimes_L K = (M \otimes_L L') \otimes_{L'} K$, and let $M' = M \otimes_L L'$ and $R' = R \otimes_L L' = R_{L'}$.  
By our assumptions the module $M'$ is still critical of dimension $d$.  Identifying $M'$ with the subset $\{ (m' \otimes 1) | m' \in M' \}$ of $M' \otimes_{L'} K = M \otimes_L K$, the argument in the previous paragraph now applies, since $N \cap M' \neq 0$ by choice of $L'$, and again shows that $N$ is finitely generated over $R_K$.  We conclude that every submodule $N$ of $M \otimes_L K$ is finitely generated, and hence $M \otimes_L K$ is noetherian. This completes the proof of the induction step.

By induction on the well-ordered poset $\Gamma$, we conclude that $P(\alpha)$ holds for all $\alpha \in \Gamma$.  In particular, the ring $R$ is a right module of finite GK dimension over $k$ by assumption, so $R \otimes_k K$ is right noetherian.  

It is now obvious that if (i) and (ii) hold for all extensions $k \subseteq K$ in the class $\mc{F}$, then $R \otimes_k K$ is noetherian for all such extensions and so $R$ is $\mc{F}$-stably right noetherian.
\end{proof}

As a first immediate application of Theorem~\ref{thm:reduce}, we show that for noetherian algebras with well-behaved GK dimension, the 
noetherian property is stable under purely transcendental field extensions.   
\begin{theorem}
\label{thm:purely-noeth}
Let $R$ be a right noetherian $k$-algebra with $\GKdim_k R = m < \infty$.  For each $d \geq 0$ write 
$K_d = k(x_1, \dots, x_d)$ and let $R_d = R \otimes_k K_d$.  Assume that for each $d \geq 0$ the GK dimension of finitely generated $R_d$-modules over the field $K_d$ is exact, finitely partitive, and takes integer values.  
Then $R$ is $\mc{PT}$-stably right noetherian; in other words, for any purely transcendental field extension $k \subseteq K$, 
$R \otimes_k K$ is right noetherian.
\end{theorem}
\begin{proof}
We aim to use Theorem~\ref{thm:reduce}, with the class $\mc{F}$ equal to the class $\mc{P} \mc{T}$ of purely transcendental extensions.  
Fix some purely transcendental extension $k \subseteq K$ and consider small fields inside $K$, in the terminology of that theorem.  In this case a small field $L$ is $k$-isomorphic to $k(x_1, \dots x_d)$ for some $d \geq 0$.  So hypothesis (i) of Theorem~\ref{thm:reduce} holds by assumption. Similarly, given an extension of small fields $L \subseteq L'$ in $\mc{F}$, then up to isomorphism we have
$L = K_d$ and $L' = K_e$ for some $e \geq d$, and $L' = L(x_{d+1}, \dots, x_e)$.  By induction on Proposition~\ref{prop:critical}, if $M$ is a critical $R_d$-module, then $M \otimes_L L'$ remains critical over $R_e$.   More generally, if $M$ is any finitely generated $R_d$-module 
with a critical composition series $0 = M_0 \subsetneq M_1 \subsetneq \dots \subsetneq M_n = M$ of length $n$, then
$0 = M_0 \otimes_L L' \subsetneq M_1 \otimes_L L' \subsetneq \dots \subsetneq M_n \otimes_L L' = M \otimes_L L'$ is 
a critical composition series of length $n$ for $M_{L'}$.  Thus we can take $n(M) = \ell(M)$ and 
hypothesis (ii) from Theorem~\ref{thm:reduce} holds as well.

Thus Theorem~\ref{thm:reduce} applies and $R$ is $\mc{P} \mc{T}$-stably right noetherian.
\end{proof}

As we noted in Section~\ref{sec:dim}, as far as we are aware there are no known examples of noetherian $k$-algebras 
of finite GK dimension for which GK dimension of finitely generated modules is not exact or not finitely partitive, and the same is true for the integer valued assumption.    So these hypotheses should be viewed as weak restrictions (or possibly not even restrictions).
Thus Theorem~\ref{thm:purely-noeth} suggests that for all reasonable noetherian algebras of finite GK dimension, 
a purely transcendental base field extension should not disrupt the noetherian property.

This paper works with algebras of finite GK dimension since this allows the special methods used in Theorem~\ref{thm:reduce} to go through.  
While many examples important in applications do have finite GK dimension, it is certainly not a necessary condition for 
being stably right noetherian in general.   
\begin{example}
Let $k$ be a field of characteristic $0$ and let $R = A_1(k)$ be the first Weyl algebra over $k$.  The algebra $R$ 
can be described as an iterated Ore extension $R = k[x][y; d/dx]$.  Then by Lemma~\ref{lem:stableprops}(1), since 
$k$ is stably right noetherian over $k$, $R$ is  stably noetherian over $k$ as well.  Then the division ring of fractions $D = Q(R)$ 
of $R$ is also stably noetherian over $k$, since by Lemma~\ref{lem:stableprops}(1) 
this property is also preserved by localization.  However, while $\GKdim(R) = 2$, we have $\GKdim(D) = \infty$, since 
$D$ contains a free subalgebra on two generators \cite{M-L}.
\end{example}
\noindent On the other hand, in the next section we will see that for polynomial identity algebras, being stably right noetherian 
does force the algebra to have finite GK dimension.  In addition, for the homologically smooth graded algebras studied in the final section,
the right noetherian property itself already implies that the algebra has finite GK dimension.

The Gabriel-Rentschler Krull dimension is a dimension function which is naturally adapted to critical composition series, which 
exist for all noetherian modules.  See \cite[Chapter 15]{GW}.  So it is tempting to try to apply the same method as in Theorem~\ref{thm:reduce} using Krull dimension, which would extend its reach.  Unfortunately, Krull dimension can easily get larger when one extends the base field, in contrast to the stability of the GK dimension under base field extension (Lemma~\ref{lem:GKbasics}(1)).  If one could show for a given $k$-algebra $R$ of finite Krull dimension that there is an upper bound on how much the Krull dimension of modules grows for field extensions in the class, the method of Theorem~\ref{thm:reduce} might go through.  Resco showed this for some special classes of algebras, 
in particular Weyl algebras and skew Laurent extensions of fields \cite[Theorem 4.2, Theorem 4.3]{Re}.  However, it appears 
to be a difficult question in general \cite[Remark 4.10]{Re}.

\section{Stably noetherian PI rings}
\label{sec:PI}

In this section we discuss the $\mc{F}$-stably noetherian property for polynomial identity (PI) algebras $R$ over a 
field $k$.  We do not necessarily assume that $R$ is affine.

The basic facts about PI rings can be found in many places, for instance \cite{DF}.  Results about the growth of PI algebras 
can be found in \cite[Chapter 10]{KL}. Suppose that $R$ is a noetherian PI $k$-algebra with $\GKdim_k(R) < \infty$.  If $R$ has minimal primes $P_1, P_2, \dots, P_m$ and nilradical $N = \bigcap P_i$, we have 
$\GKdim_k(R) = \GKdim_k(R/N) = \max \{ \GKdim_k (R/P_i) | 1 \leq i\leq m \}$ \cite[Proposition 10.15]{KL}.   If $M$ is a finitely generated right $R$-module, then $\GKdim_k(M) = \GKdim_k(R/I)$ where $I = \operatorname{r.ann}(M)$ \cite[Lemma 10.18]{KL}.  The value of $\GKdim_k(R/I)$ is an integer \cite[Corollary 10.16]{KL}, and thus $\GKdim_k(M)$ is an integer for any finitely generated right $R$-module $M$.  Moreover, GK dimension is exact for finitely generated $R$-modules \cite[Proposition 10.20]{KL} and finitely partitive \cite[Corollary 10.22]{KL}.   The value of $\GKdim_k(R)$ 
is also related to the theory of transcendence degree:  if $R$ is a prime PI algebra with classical ring of fractions $Q$, then if $Z = Z(Q)$ is its center we have $\GKdim_k(R) = \operatorname{tr.deg}_k(Z)$ \cite[Theorem 10.5]{KL}.   Since GK dimension of modules is integer valued, exact, and finitely partitive, every finitely generated $R$-module $M$ has a critical composition series with respect to GK dimension.

\begin{definition}
Let $\mc{F}$ be an extensive class of field extensions.  We say that the right noetherian PI $k$-algebra 
$R$ has \emph{right $\mc{F}$-stable residues} over $k$ if the classical ring of fractions $Q(R/P)$ of $R/P$ is $\mc{F}$-stably right noetherian over $k$, for all prime factor rings $R/P$ of $R$. 
\end{definition}

The terminology alludes to the commutative case, where the rings $Q(R/P)$ are exactly 
the residue fields of the local rings $R_P$ of points $P$ in $\spec R$.
It is clear from Lemma~\ref{lem:stableprops}(1) that if $R$ is $\mc{F}$-stably right noetherian over $k$, then it has right $\mc{F}$-stable residues over $k$.  We are going to show that the converse holds for PI algebras of finite GK dimension.  In fact, we know of no example where the converse does not hold, though it is especially useful in the PI case since PI division rings have such a simple structure. 

First, we give some basic properties of the notion of having right $\mc{F}$-stable residues.
\begin{lemma}
\label{lem:stable-res-fields}
Let $\mc{F}$ be a extensive class of field extensions.   Let $R$ be a right noetherian $k$-algebra with right $\mc{F}$-stable residues over $k$.
\begin{enumerate}
\item  Every factor ring of $R$, localization of $R$ at a right denominator set, and polynomial extension $R[x]$ 
also has right $\mc{F}$-stable residues over $k$.  

\item If $k \subseteq \ell$ is a finitely generated field extension in the class $\mc{F}$, then $R \otimes_k \ell$ has right $\mc{F}$-stable residues over $\ell$.
\end{enumerate}
\end{lemma}
\begin{proof}
It is obvious by definition that $R/I$ has right $\mc{F}$-stable residues over $k$ for any ideal $I$.
Consider next a localization $RX^{-1}$ where $X$ is a right denominator set in $R$.  Any prime ideal $P$ of the ring $RX^{-1}$ has the form $P = NX^{-1}$ for some prime ideal $N$ of $R$ such that $N \cap X = \emptyset$, 
and $Q(RX^{-1}/P) \cong Q(R/N)$ \cite[Theorem 10.20, Exercise 10O]{GW}.  So $RX^{-1}$ also has right $\mc{F}$-stable residues over $k$.  Finally, consider the polynomial ring $R[x]$.
If $N \subseteq R[x]$ is a prime ideal then let $R' = R/(N \cap R)$, where $N \cap R$ is prime in $R$.  Thus $R[x]/N \cong R'[x]/N'$ where $N' = N/((N \cap R)[x])$.  We have seen that $R'$ also has right $\mc{F}$-stable residues, so it is enough to work with the prime ring $R'$ and the prime ideal $N'$ of $R'[x]$, where $N' \cap R' = 0$.  Changing notation back, we assume that 
$R$ is prime, and that we have a prime ideal $N \subseteq R[x]$ such that $N \cap R = 0$.   
Now let $X$ be the set of regular elements in $R$ and let $S = Q(R) = RX^{-1}$ be 
the classical ring of fractions of $R$.  Consider $R[x]X^{-1} \cong S[x]$.  Since $N \cap X = \emptyset$, 
$NX^{-1}$ is a prime ideal of $S[x]$, and $Q(R[x]/N) \cong Q(S[x]/NX^{-1})$.
By the assumption that $R$ has right $\mc{F}$-stable residues over $k$, $S$ is $\mc{F}$-stably right noetherian over $k$.   Then $Q(S[x]/NX^{-1})$ is also $\mc{F}$-stably right noetherian over $k$, by Lemma~\ref{lem:stableprops}(1).  
So $R[x]$ has right $\mc{F}$-stable residues.

(2). We can write $\ell$ as a localization of a finitely generated $k$-algebra $k[x_1, \dots, x_n]/I$, and so 
$T = R \otimes_k \ell$ is a localization of $R[x_1, \dots, x_n]/I R[x_1, \dots, x_n]$.  By part (1), $T$ has 
$\mc{F}$-stable residues over $k$.  Then for every prime $P$ of $T$, the quotient ring $Q(T/P)$ is $\mc{F}$-stably right noetherian over $k$.
By Lemma~\ref{lem:stableprops}(3),  $Q(T/P)$ is also $\mc{F}$-stably right noetherian over $\ell$.
So $T$ has right $\mc{F}$-stable residues over $\ell$.
\end{proof}

\begin{theorem}
\label{thm:reduce-to-residues}
Let $\mc{F}$ be a extensive class of field extensions of $k$.  Let $R$ be a right noetherian PI $k$-algebra with $\GKdim_k(R) < \infty$.  Then $R$ has right $\mc{F}$-stable residues if and only if $R$ is $\mc{F}$-stably right noetherian.
\end{theorem}
\begin{proof}
If $R$ is $\mc{F}$-stably right noetherian, then it is immediate that $R$ has right $\mc{F}$-stable 
residues by Lemma~\ref{lem:stableprops}(1).

Now assume that $R$ has right $\mc{F}$-stable residues over $k$.  To prove that $R$ is $\mc{F}$-stably right noetherian, we apply Theorem~\ref{thm:reduce}.  Fix an extension $k \subseteq K$ in $\mc{F}$ and consider a small field $L$; recall this 
means that  $k \subseteq L \subseteq K$ with $k \subseteq L$ and $L \subseteq K$ in $\mc{F}$ and with $k \subseteq L$ a finitely generated 
extension.  Then $R_L = R \otimes_k L$ is also noetherian.  The ring $R_L$ is also still a PI $L$-algebra (for example, since $R$ 
must satisfy a multilinear polyomial identity \cite[Remark A.1.2.10]{DF}, and such an identity continues to be  
an identity for $R_L$).  Also, $\GKdim_L(R_L) = \GKdim_k(R) < \infty$, and hence GK dimension of $R_L$-modules over $L$ is exact and partitive and takes integer values.  So Hypothesis (i) of Theorem~\ref{thm:reduce} holds.

To vertify Hypothesis (ii), we need to show, for all finitely generated right $R_L$-modules $M$, that there is an upper bound $n(M)$ on the lengths of the critical composition series of $M' = M \otimes_L L'$, for all small $L'$ with $L \subseteq L'$ in the class $\mc{F}$.  Consider the set 
$\mc{B}$ of all such $R_L$-modules $M$  for which the bound $n(M)$ exists; we need to prove that $\mc{B}$ contains all finitely generated $R_L$-modules.
It follows from Lemma~\ref{lem:ccfacts}(2) that given a short exact sequence $0 \to M \to S \to N \to 0$ of finitely generated 
right $R_L$-modules, if $M$ and $N$ are in $\mc{B}$ then so is $S$ (and in fact we can take $n(S) = n(M) + n(N)$).
Recall that a \emph{prime module} is a right module $M$ such that $\operatorname{r.ann}_R(M) = P$ is prime, and where every nonzero submodule of $M$ 
also has annihilator $P$.  Since a finitely generated right $R_L$-module $M$ is noetherian, it has a prime series \cite[Proposition 3.13]{GW}; that is, a sequence of submodules $0 = M_0 \subsetneq M_1 \subsetneq \dots \subsetneq M_n = M$ for which each $M_i/M_{i-1}$ is a prime module for some prime $P_i$, for all $1 \leq i \leq n$.  Using this, it suffices to prove that each prime module is in $\mc{B}$.  So now let $M$ be a prime $R_L$-module with $\operatorname{r.ann}_{R_L}(M) = P$.  Let $T = R_L/P$.  Given any nonzero submodule $M'$ of $M$, by we have $\operatorname{r.ann}_{R_L}(M') = P$ and so $\GKdim(M') = \GKdim(R_L/P)$ (using the PI property).  Thus $M$ is $d$-homogeneous, where $d = \GKdim_L(T)$.  

Now $M_{L'} = M \otimes_L L'$ is a $d$-homogeneous $T_{L'} = T \otimes_L L'$-module, by Proposition~\ref{prop:homogeneous}.  Suppose it has a critical composition series  $0 = M_0 \subsetneq M_1 \subsetneq \dots \subsetneq M_n = M_{L'}$ over $T_{L'}$.  Then each 
factor $M_i/M_{i-1}$ is $d$-critical.  Note that $\GKdim_{L'}(T_L') = \GKdim_L(T) = d$.  Now since $T$ is a prime PI algebra, its center $Z$ is a domain and its PI classical ring of fractions $Q(T)$ is equal to $TX^{-1}$, where $X= Z \setminus \{0 \}$, by Posner's theorem \cite[Theorem B.6.5]{DF}.  Note that $Y = \{ (x \otimes 1) | x \in X \}$ consists of central regular elements for the ring $T \otimes_L L'$,  and that $(T \otimes_L L')Y^{-1} \cong Q(T) \otimes_L L'$.
Then the series $0 = M_0 Y^{-1} \subsetneq M_1 Y^{-1} \subsetneq \dots \subsetneq M_n Y^{-1} = MX^{-1} \otimes_L L'$ 
has factors $(M_i/M_{i-1}) Y^{-1}$ which are $d$-critical $(T \otimes_L L') Y^{-1} = Q(T) \otimes_L L'$-modules, 
by Lemma~\ref{lem:loc-crit}(2), and so 
this is a critical composition series for $MX^{-1} \otimes_L L'$.  Now the further base field extension 
$MX^{-1} \otimes_L L' \otimes_{L'} K = MX^{-1} \otimes_L K$ is a finitely generated module over $TX^{-1} \otimes_L K = Q(T) \otimes_L K$.
Since $R$ has $\mc{F}$-stable residues over $k$, $R \otimes_k L$ has $\mc{F}$-stable residues over $L$, by Lemma~\ref{lem:stable-res-fields}(2).  So $Q(T) \otimes_L K$ is right noetherian.  In particular, $MX^{-1} \otimes_L K$  has a critical composition series of some length $\ell$, which is greater than or equal to the length $n$ of the critical composition series for $MX^{-1} \otimes_L L'$ by Lemma~\ref{lem:cs-ext}.  It follows that we can take $n(M) = \ell$ to satisfy hypothesis (ii). 

Now by Theorem~\ref{thm:reduce}, $R \otimes_k K$ is right notherian, and since $k \subseteq K$ is an arbitrary extension 
in the class, $R$ is $\mc{F}$-stably right noetherian.
\end{proof}

Because simple PI rings are matrix rings over division rings which are finite over their centers, it is easy to understand when they are stably right noetherian.  Thus using the previous result, we can now completely characterize stably right noetherian PI algebras.
\begin{corollary}
\label{cor:PI-char}
Let $R$ be a right noetherian PI $k$-algebra.  Then $R$ is stably right noetherian over $k$ if and only if for all prime ideals $P$ of $R$, the center $Z$ of $Q(R/P)$ is a finitely generated field extension of $k$. 
\end{corollary}
\begin{proof}
Consider a prime ideal $P$ of $R$.  Since $R/P$ is a prime PI algebra, so is its classical ring of fractions $Q(R/P)$, again by Posner's Theorem \cite[Theorem B.6.5]{DF}.  So $Q(R/P)$ is a simple PI ring, and thus it is a matrix ring over a division ring $D$ which
is finite over its center $Z$, by Kaplansky's theorem \cite[B.4.13]{DF}.  Since $Q(R/P)$ is finite and free over $Z$, $Q(R/P)$ is 
stably right noetherian over $k$ if and only if $Z$ is, using Lemma~\ref{lem:stableprops}(1,2).  Since $Z$ is a field, it is stably right noetherian over $k$ if and only if it is a finitely generated field extension of $k$, by Theorem~\ref{thm:vamos} and Lemma~\ref{lem:stableprops}(1).

Now if $R$ is stably right noetherian over $k$, we know that $Q(R/P)$ is stably right noetherian over $k$ for each $P$ by Lemma~\ref{lem:stableprops}(1), and so its center $Z$ is a finitely generated field extension of $k$, by the first paragraph.  Conversely, assume that the center $Z$ of each $Q(R/P)$ is a finitely generated field extension of $k$.  In particular, for each minimal prime $P$ of $R$, $\GKdim_k(R/P) <\infty$ (since this is equal to the transcendence degree of the center $Z$ of $Q(R/P)$ over $k$, which must be finite).  Thus $\GKdim_k(R) < \infty$.  By the first paragraph, $R$ has right $\mc{F}$-stable residues over $k$, where $\mc{F}$ is the class of all field extensions.  Then by Theorem~\ref{thm:reduce-to-residues}, $R$ is stably right noetherian over $k$.
\end{proof}

The result above is nontrivial already in the commutative case, in which the statement simplifies as follows:
\begin{corollary}
\label{cor:comm-char}
Let $R$ be a commutative noetherian $k$-algebra.  Then $R$ is stably right noetherian over $k$ if and only if for all prime ideals $P$ of $R$, the fraction field $Q(R/P)$ is a finitely generated field extension of $k$. 
\end{corollary}
\noindent  This corollary answers \cite[Question 2.55]{Fa} in Farina's work.

It is natural to wonder if one really has to look at all prime factor rings in the characterizations of Corollary~\ref{cor:PI-char} 
and Corollary~\ref{cor:comm-char}.  
As pointed out by Farina in \cite[Section 2.5.2]{Fa}, Wadsworth has given a example of a commutative domain $R$
with quotient field $Q(R)$ which is a finitely generated extension of $k$, but where $R$ has a maximal ideal 
$P$ such that $Q(R/P) = R/P$ is not a finitely generated field extension of $k$ \cite[Remark 3 following Example 1]{Wa}.  Thus $R$ fails to be stably noetherian by Corollary~\ref{cor:PI-char}, but the failure is only seen by looking at a factor ring by a non-minimal prime.

On the other hand, the characterization of $\mc{PT}$-stably right noetherian PI algebras is simpler, and does not require 
looking at all prime factor rings.
\begin{corollary}
\label{cor:comm-purely-noeth}
Let $R$ be a right noetherian PI $k$-algebra.  Then $R$ is $\mc{PT}$-stably right noetherian over $k$ if and only if $\GKdim_k R < \infty$.
\end{corollary}
\begin{proof}
Suppose that $\GKdim_k R < \infty$.  For any purely transcendental field extension $k \subseteq K_d = k(x_1, \dots, x_d)$, $R_d = R \otimes_k K_d$ is again a right noetherian PI $K_d$-algebra with finite GK dimension, so in particular the GK dimension of $R_d$-modules over $K_d$ is exact, finitely partitive, and integer valued.  Thus by Theorem~\ref{thm:purely-noeth}, $R$ is $\mc{PT}$-stably right noetherian over $k$.  

Conversely, if $R$ is $\mc{PT}$-stably noetherian over $k$, so is $Q(R/P)$ for each prime $P$ of $R$, 
by Lemma~\ref{lem:stableprops}(1).  Let $Z$ be the center of $Q(R/P)$.  If $Z$ has infinite transcendence degree over $k$, there is a field embedding $L = k(x_i | i \in \mb{N}) \subseteq Z$; since $Q(R/P)$ is free over $L$, $L$ must also be 
$\mc{PT}$-stably noetherian over $k$, by Lemma~\ref{lem:stableprops}(2).   But by V{\'a}mos's theorem (Theorem~\ref{thm:vamos}), $L \otimes_k L$ is not noetherian, a contradiction.  Thus $\GKdim(R/P) = \trdeg_k Z < \infty$.  Recalling that $\GKdim(R)$ is 
the maximum of $\GKdim(R/P)$ over minimal primes $P$ of $R$, we get $\GKdim_k(R) < \infty$.
\end{proof}

We note that the characterizing properties found in Corollary~\ref{cor:PI-char} and Corollary~\ref{cor:comm-purely-noeth} are left-right symmetric.  It is clear that these results (and all of the results in this paper) have analogous left-sided versions, so we conclude 
that a noetherian PI algebra is stably right noetherian if and only if it is stably left noetherian, and is 
$\mc{P} \mc{T}$-stably right noetherian if and only if it is $\mc{P} \mc{T}$-stably left noetherian.  By contrast, as is well-known,  
there are many PI algebras that are noetherian only on one side.

\section{Graded algebras}
\label{sec:graded}

In this last section we study the stably noetherian property for $\mb{N}$-graded $k$-algebras $A = \bigoplus_{n \geq 0} A_n$.
We say that a $\mb{Z}$-graded right $A$-module $M = \bigoplus_{n \in \mb{Z}} M_n$ is \emph{locally finite} if 
$\dim_k M_n < \infty$ for all $n \in \mb{Z}$, and \emph{left bounded} if $M_n = 0$ for all $n \ll 0$.  If $A$ is locally finite (as a module over itself), then every finitely generated graded right $A$-module $M$ is left bounded and locally finite.  
  We say that $A$ is \emph{connected} if $A_0 = k$.  While the results of this section are interesting already for connected algebras, using recent results from \cite{RR2} it is not much more difficult to handle the case of general locally finite graded algebras $A$, and so we state our main results in that setting. For a locally finite graded algebra $A$, let $J(A)$ be its graded Jacobson radical, in other words the intersection of all maximal graded right ideals.  It is easy to see that $A_{\geq 1} \subseteq J(A)$, and therefore that $A/J(A) \cong A_0/J(A_0)$, where $J(A_0)$ is the usual Jacobson radical of the Artinian algebra $A_0$.  In particular, $S = A/J(A)$ is finite dimensional over $k$  and semisimple.  
We say that $A$ is \emph{elementary} if $S = k^m$ for some $m$.  The reader can find more details in \cite[Section 2]{RR1}.

Let $A$ be an $\mb{N}$-graded locally finite $k$-algebra.  Let $A^e = A \otimes_k A^{op}$ be the \emph{enveloping algebra} of $A$.
We may identify $(A, A)$-bimodules with left $A^e$-modules.  In particular, 
$A$ has a canonical left $A^e$-module structure, and we say that $A$ is \emph{(graded) homologically smooth} if $A$ has a finite 
length graded projective resolution by finitely generated graded projective $A^e$-modules.  A homologically smooth algebra must have finite global dimension.  Indeed, if $P_{\bullet}$ is the finite length $A^e$-projective resolution of $A$, then for any right $A$-module $M$, the complex $M \otimes_A P_{\bullet}$ gives a finite length projective resolution of $M$ \cite[Lemma 3.2(1)]{RR1}.  However, homological smoothness is slightly stronger for finitely graded algebras than finite global dimension, in a useful way.  
\begin{lemma} \cite[Corollary 3.22, Corollary 3.14]{RR1}
\label{lem:gldim}
Let $A$ be a locally finite $\mb{N}$-graded $k$-algebra with graded Jacobson radical $J(A)$, and  let $S = A/J(A)$.
\begin{enumerate}
\item If $A$ is right noetherian, then $A$ is homologically smooth if and only if $A$ has finite global dimension and $S$ is a separable $k$-algebra.
\item If $A$ is homologically smooth over $k$, then $A_K = A \otimes_k K$ is also homologically smooth over $K$ for any field 
extension $k \subseteq K$; in particular $A_K$ still has finite global dimension.
\end{enumerate}
\end{lemma}
Recall that a finite-dimensional $k$-algebra $S$ is \emph{separable} over $k$ if $S \otimes_k K$ is semsimple for 
all field extensions $k \subseteq K$.  Thus if $A$ is right noetherian $\mb{N}$-graded and elementary, then 
$S = k^m$ is automatically separable over $k$, and so part (1) of the preceding result shows that homological smoothness is just the same as finite global dimension.  But for arbitrary locally finite graded algebras, the additional condition that $S$ is separable is important; in contrast 
to part (2) of the lemma, finite global dimension is not preserved by base field extension in general.  For example, 
if $k \subseteq E$ is a finite degree purely inseparable field extension and we take $A = E = A_0$, then $A$ has finite global dimension but $A \otimes_k E$ has infinite global dimension.  See \cite[Example 3.20]{RR1}.

For a left bounded locally finite graded $A$-module $M$, its \emph{Hilbert series} is the formal Laurent series 
$h_M(t) = \sum_{n \in \mb{Z}} (\dim_k M_n) t^n \in \mb{Z}((t))$.  We say that the Hilbert series of $M$ is \emph{rational} 
if $h_M(t) = p(t)/q(t)$ for some $p(t) \in \mb{Z}[t, t^{-1}]$, $q(t) \in \mb{Z}[t]$ with $p(t)$ and $q(t)$ relatively prime and $q(0) = \pm 1$.  In this case, since $\mb{Q}[t] = \mb{Q}[1-t]$ we can also write $h_M(t)$ as a Laurent series in powers of $(1-t)$, say $h_M(t) = a_{-d} (1-t)^{-d} + a_{-d+1} (1-t)^{-d+1} + \dots$, for some $d \in \mb{Z}$ where $a_{-d} \neq 0$, so $d$ is the order of the pole of $h_M(t)$ at $t = 1$.   We call $a_{-d}$ the \emph{multiplicity} of $r(t)$ and write it as $\epsilon(r(t))$.   If $\GKdim(M) < \infty$, then $\GKdim(M) = d$ is the order of the pole of $h_M(t)$ at $t = 1$ (see \cite[Lemma 2.7]{RR2}).

The theory of Hilbert series is especially nice for elementary homologically smooth graded algebras; see \cite[Section 4]{RR2}.
The most important fact for us here is the following granularity of the multiplicities of finitely generated modules.
\begin{lemma}
\label{lem:granular}
Let $A$ be a locally finite $\mb{N}$-graded $k$-algebra which is homologically smooth and right noetherian.  Assume that $A$ is elementary, so $S = A/J(A) = k^m$ for some $m$.  

There is a number $\epsilon > 0$, depending only on $A$, such that given any finitely generated field extension $k \subseteq L$ and a finitely generated graded right $A \otimes_k L$-module $M$, then $M$ has a rational Hilbert series over $L$ and $\epsilon(h_M(t))$ is a positive integer multiple of $\epsilon$.
\end{lemma}
\begin{proof}
The hypotheses of homologically smooth and right noetherian imply that $\GKdim(A) < \infty$ \cite[Proposition 4.4]{RR2}.  We can decompose $1 = e_1 + \dots + e_m$, where the $e_i$ are pairwise orthogonal primitive idempotents $e_i \in A_0$.  The \emph{matrix Hilbert series} of $A$ is the formal series $\sum_{n \geq 0} H_n t^n \in M_m(\mb{Z}[[t]])$, where $H_n$ is the matrix with $(i,j)$-entry $\dim_k e_i A_n e_j$.
The matrix Hilbert series of $A$ is of the form $q(t)^{-1}$, for some matrix polynomial $q(t) \in M_m(\mb{Z}[t])$ such 
that $D(t) = \det q(t)$ satisfies $D(0) = \pm 1$ \cite[Proposition 4.2(1)]{RR2}.  Recall that an $A$-module $M$ is \emph{perfect} 
if it has a finite length projective resolution where each projective is finitely generated.  By \cite[Proposition 4.2(2,3)]{RR2}, 
if $M$ is a finitely generated graded right $A$-module which is perfect with $\GKdim(M) < \infty$, then $h_M(t)$ is rational 
with multiplicity $\epsilon(h_M(t))$ which is positive and an integer multiple of $\epsilon(D(t)^{-1})$.  Since we assume that $A$ is noetherian
and the homologically smooth assumption implies that $A$ has finite global dimension, every finitely generated graded right $A$-module $M$ is perfect.  Moreover, $\GKdim(M) < \infty$ since $\GKdim(A) < \infty$.  So every finitely generated graded right $A$-module $M$ has 
a multiplicity $\epsilon(h_M(t))$ which is a positive integer multiple of $\epsilon$, where $\epsilon = |\epsilon(D(t)^{-1})| >0$.

When we extend the base field by a finitely generated extension $k \subseteq L$, then $A_L$ is still right noetherian 
and homologically smooth over $L$, by Lemma~\ref{lem:gldim}(2).
We claim that the idempotents $e_i$ are still primitive in $A_L$.  Note that $S$ is separable over $k$ since 
$A$ is homologically smooth,  by Lemma~\ref{lem:gldim}(1).  It follows that 
$A_L/J(A_L) = (A/J(A)) \otimes_k L = k^m \otimes_k L \cong L^m$ \cite[Lemma 3.7]{RR1}; in particular, $A_L$ is elementary over $L$.  If the $e_i$ are not all still primitive, then $1$ decomposes
as a sum of more than $m$ idempotents in $A_L$, contradicting $A_L/J(A_L) = L^m$.  This proves the claim, and so 
$1$ has the same decomposition as a sum of primitive orthogonal idempotents in $A_L$ as in $A$.  It follows that $A_L$  has the same matrix Hilbert series $q(t)^{-1}$ over $L$.   Thus the same argument as in the previous paragraph applies to finitely generated graded modules over $A_L$, with the same constant $\epsilon$.
\end{proof}

We conclude with our main result on the stably noetherian property for graded algebras.  
\begin{theorem}
\label{thm:graded-stable}
Let $A$ be a locally finite right noetherian $\mb{N}$-graded $k$-algebra which is homologically smooth over $k$. Then $A$ is stably right noetherian over $k$.
\end{theorem}
\begin{proof}
We first reduce to the elementary case.  By Lemma~\ref{lem:gldim}(1), $S = A/J(A)$ is separable over $k$.  Then by 
\cite[Theorem 4.3]{RR2}, there is a finite degree field extension $k \subseteq \ell$ and a full idempotent $e \in A_{\ell} = A \otimes_k \ell$ such that $A' = e A_{\ell} e$ is an elementary locally finite graded algebra over $\ell$.  By \cite[Corollary 3.14, Proposition 3.17]{RR1}, $A'$ is still homologically smooth over $\ell$.  Suppose we prove that $A'$ is stably right noetherian over $\ell$.
Then for any extension $\ell \subseteq K$, $eA_{\ell}e \otimes_\ell K$ is right noetherian.  Since $e$ is a full idempotent in 
$A_{\ell}$ (that is, $A_{\ell} e A_{\ell} = A_{\ell}$), $e \otimes 1$ is also a full idempotent in $A_{\ell} \otimes_{\ell} K$, 
and so $(e \otimes 1) (A_{\ell} \otimes_{\ell} K ) (e \otimes 1) \cong (e A_{\ell} e) \otimes_{\ell} K$ is Morita equivalent to $A_{\ell} \otimes_{\ell} K$.  Thus $A_{\ell} \otimes_{\ell} K$ is also right noetherian, since this is a Morita invariant property.  This shows that $A_{\ell}$ is stably right noetherian over $\ell$.  Then $A$ will be stably noetherian over $k$ by Lemma~\ref{lem:stableprops}(4).  Thus it is enough to prove the theorem for elementary algebras.  

We now change notation back and assume that $A$ itself is elementary as a $k$-algebra.  By \cite[Proposition 4.4]{RR2}, 
$\GKdim_k(A) < \infty$.  We now aim to apply Theorem~\ref{thm:reduce}, but we apply it in the graded category.  Thus we work with graded GK critical modules (graded modules such that every nontrivial factor module by a graded submodule has smaller GK dimension) and graded critical composition series, and consider the graded noetherian property of modules (ACC on chains of graded submodules).  It is easy to see that the theorem works in this graded setting as well, with the same proof.
  
We verify Hypotheses (i) and (ii) from Theorem~\ref{thm:reduce} in the graded case.  Fix a field extension $k \subseteq K$ 
and let $k \subseteq L \subseteq K$ with $k \subseteq L$ a finitely generated extension.  Then $A_L$ is still noetherian and a homologically smooth elementary $L$-algebra, by Lemma~\ref{lem:gldim}(2).  By \cite[Proposition 2.8, Proposition 4.4]{RR2}, GK dimension of finitely generated graded $A_L$-modules is graded exact, graded finitely partitive, and integer valued.    Thus hypothesis (i) holds.

For (ii), we need to show that if $k \subseteq L \subseteq L' \subseteq K$ with $k \subseteq L'$ still finitely generated, given a nonzero finitely generated graded $R_L$-module $M$, there is an upper bound on the lengths of the graded critical composition series of the modules $M' = M \otimes_L L'$, as $L'$ varies.   As we saw in the proof of Theorem~\ref{thm:reduce-to-residues}, the set of modules with this property is closed under extensions, so it is enough to prove it for an $R_L$-module $M$ which is itself graded critical, say with $\GKdim_L(M) = d$.  Then $M$ is $d$-homogeneous and so $M'$ is still $d$-homogeneous by Proposition~\ref{prop:homogeneous}.  Hence the graded critical composition series of $M'$ has factors all of which are $d$-critical.  Using Lemma~\ref{lem:granular}, there is a fixed $\epsilon > 0$ such that the multiplicity of any 
nonzero finitely generated graded $R_L$ or $R_{L'}$-module is a positive multiple of $\epsilon$.  In particular, 
$\epsilon(h_M(t)) = c \, \epsilon$ for some integer $c > 0$.  Since the Hilbert series of $M'$ over $L'$ is the same as the Hilbert series of $M$ over $L$, we also have $\epsilon(h_{M'}(t)) = c \, \epsilon$.  But multiplicity is additive in short exact graded sequences of graded modules, all of which have the same dimension $d$ \cite[Proposition 4.2(4)]{RR2}.  Thus $c$ is an upper bound for 
the length of a graded critical composition series for $M'$.  Since the number $c$ does not depend on the extension $L'$,
we can take $n(M) = c$ to verify Hypothesis (ii).  

Now by Theorem~\ref{thm:reduce} (applied in the graded category), all finitely generated graded right $R$-modules are stably graded right  noetherian.  In particular, $R$ is.  Thus $R \otimes_k K$ is graded right noetherian for all field extensions $k \subseteq K$.  However, for locally finite $\mb{N}$-graded algebras, graded right noetherian implies right noetherian \cite[Proposition II.3.1]{NV}.  Thus $R \otimes_k K$ is right noetherian for all extensions $k \subseteq K$, and $R$ is stably noetherian over $k$.
\end{proof}

Suppose that $A$ is a locally finite $\mb{N}$-graded $k$-algebra.  We say 
that $A$ is \emph{(graded) twisted Calabi-Yau of dimension $d$} if it is homologically smooth and in addition $\Ext^d_{A^e}(A, A^e) \cong U$ as graded right $A^e$-modules, for some graded invertible $(A,A)$-bimodule $U$, with $\Ext^i_{A^e}(A, A^e) = 0$ for $i \neq d$.  
So Corollary~\ref{cor:cy} follows directly from Theorem~\ref{thm:graded-stable} and the definition.  
Twisted Calabi-Yau algebras been of much interest in recent years.   As a special case, a connected graded twisted Calabi-Yau algebra is the same as an Artin-Schelter regular algebra (in the version of the definition which does not assume finite GK dimension) \cite[Lemma 1.2]{RRZ}. 
In particular, any member of one of these important classes of algebras is homologically smooth, and so our theorem 
shows that if it is noetherian, then it must be stably noetherian.  On the other hand, it is an open question 
whether graded twisted Calabi-Yau algebras with finite GK dimension must be noetherian; even for AS regular algebras this is known to be true 
only in small dimension.  

\bibliographystyle{plain}

\end{document}